\documentclass[reqno]{amsart}
\usepackage[T1]{fontenc}
\usepackage{dsfont}
\usepackage{mathrsfs}
\usepackage[colorlinks]{hyperref}
\usepackage{xcolor}
\usepackage[a4paper,asymmetric]{geometry}
\usepackage{mathscinet}
\usepackage{latexsym}
\usepackage{amsthm}
\usepackage{amssymb}
\usepackage{amsfonts}
\usepackage{amsmath}
\usepackage{longtable}
\usepackage{accents}
\usepackage{graphicx}
\usepackage{multirow}
\usepackage{multicol}
\usepackage{empheq}

\newtheorem{theorem}{Theorem}[section]
\newtheorem{thm}[theorem]{Theorem}

\newtheorem{lem}[theorem]{Lemma}
\newtheorem{remark}[theorem]{Remark}
\newtheorem{proposition}[theorem]{Proposition}
\newtheorem{prop}[theorem]{Proposition}
\newtheorem{corollary}[theorem]{Corollary}

\theoremstyle{definition}

\newtheorem{defn}[theorem]{Definition}

\theoremstyle{remark}
\numberwithin{equation}{section}

 \DeclareMathAlphabet{\mathpzc}{OT1}{pzc}{m}{it}
 \DeclareMathAlphabet{\mathsfsl}{OT1}{cmss}{m}{sl}

\newcommand{\dif}{\mathrm{d}}

\newcommand{\abs}[1]{\left\vert#1\right\vert}
\newcommand{\set}[1]{\left\{#1\right\}}

\newcommand{\norm}[1]{\left\Vert#1\right\Vert}

\newcommand{\loc}{{\text{loc}}}

\newcommand{\RR}{\mathbb{R}}

 \newcommand{\Rnum}{\mathbb{R}}
 
 \newcommand{\Znum}{\mathbb{Z}}

\newcommand{\Be}{\begin{equation}}
\newcommand{\Ee}{\end{equation}}
\newcommand{\Bs}{\begin{split}}
\newcommand{\Es}{\end{split}}
\newcommand{\Bes}{\begin{equation*}}
\newcommand{\Ees}{\end{equation*}}
\newcommand{\BT}{\begin{thm}}
\newcommand{\ET}{\end{thm}}
\newcommand{\Bp}{\begin{proof}}
\newcommand{\Ep}{\end{proof}}
\newcommand{\BL}{\begin{lem}}
\newcommand{\EL}{\end{lem}}
\newcommand{\BP}{\begin{proposition}}
\newcommand{\EP}{\end{proposition}}
\newcommand{\BC}{\begin{corollary}}
\newcommand{\EC}{\end{corollary}}
\newcommand{\BR}{\begin{remark}}
\newcommand{\ER}{\end{remark}}
\newcommand{\BD}{\begin{defn}}
\newcommand{\ED}{\end{defn}}
\newcommand{\BI}{\begin{itemize}}
\newcommand{\EI}{\end{itemize}}

\allowdisplaybreaks

\begin{document}
\title[Gradient and Stability estimates of Fractional Operators]{ Gradient   and Stability Estimates of  Heat Kernels for Fractional Powers of Elliptic Operator }
\author[Y. Chen]{Yong Chen}
\address{School of Mathematics, Hunan University of Science and Technology, Xiangtan, 411201, Hunan, China}
\email{zhishi@pku.org.cn; chenyong77@gmail.com}
\author[Y. Hu]{Yaozhong Hu}
\address{Department of Mathematics, the University of Kansas, Lawrence, 66045, Kansas,USA}
\email{yhu@ku.edu}
\author[Z. Wang]{Zhi Wang}
\address{School of Sciences, Ningbo University of Technology,
 Ningbo 315211, Zhejiang, China}
\email{wangzhi1006@hotmail.com}
\begin{abstract}
Gradient  and      stability type estimates of heat kernel
associated with  fractional power of a uniformly elliptic operator
are obtained.    $L^p$-operator norm of  semigroups
associated with fractional power of two  uniformly elliptic operators
 are also obtained.
 \end{abstract}

\maketitle

\section{Introduction and main conclusions}
Let $\mathbf{D}$ be a domain in $\RR^d$ and let $a:\RR^d\rightarrow \RR^{d^2}$ be a
matrix valued function with $C^\beta$ or measurable entries. The operator
$H=\nabla (a(x) \nabla )$ generated a semigroup $P_t$ which is given
by $P_tf(x)=\int_{\RR^d} p_t(x,y) f(y)dy$.  Heat kernel, gradient and stability estimates  associated with this semigroup are well-studied (see \cite{chz 98},
\cite{davies 1990}, \cite{stroock 1988}). In this paper we are concerned with
the similar estimates for the  semigroup generated by the
fractional powers of $H$, namely,
$Q_t=e^{-t(-H)^{\alpha}}$,  where $\alpha\in (0, 1)$
will be fixed throughout this paper. Our motivation
 is   recent
works  on fractional diffusion  in random environment (see \cite{chen 2007, hh} and references therein) arisen from super and sub diffusion  in random environment.
However, we shall  deal with this problem in separate project.

First let us recall  a result.
Using the classical Bromwich contour integral, Pollard in \cite{pollard 1946}    obtained  the following formula for the  inverse Laplace transform of  the function $e^{-u^{\alpha}}$.   
\begin{align}
   e^{-u^{\alpha}}&=\int_{0}^{\infty}e^{-u s} g(\alpha,\, s)\dif s,\qquad u\ge 0.
\label{invers formula}
\end{align}
where
\begin{align}
   g(\alpha,\, s)&=\frac{1}{\pi}\int_0^{\infty} e^{-s u}e^{-u^{\alpha} \cos \pi\alpha}\sin(u^{\alpha}\sin \pi\alpha)\,\dif u,\quad s\ge 0.
   \label{invers formula g form}
\end{align}
is a probability density function of $s\ge 0$.
This class of density functions $g(\alpha,s)$ is called strictly $\alpha$-stable law which plays important role in the theory of probability.

Denote
 \begin{equation}\label{g t x}
 \lambda_t(ds)=g_{_t}(\alpha,\, s)ds:= t^{-\frac{1}{\alpha}}g(\alpha,\, t^{-\frac{1}{\alpha}}s)ds.
 \end{equation}
 Then
 \begin{eqnarray*}
 e^{-u^{\alpha} t }
 &=&e^{-(ut^{1/\alpha})^\alpha}
 =\int_{0}^{\infty}e^{-ut^{1/\alpha}  s} g(\alpha,\, s)\dif s\\
 &=& \int_{0}^{\infty}e^{-u   s} g_t(\alpha,\, s)\dif s\\
 \end{eqnarray*}


From this identity we can define the semigroup
$Q_t=e^{-t(-H)^\alpha}$ associated with $-(-H)^\alpha$ as
\begin{equation}\label{qtof f}
  Q_tf(x)=\int_0^{\infty}P_sf(x) \,\lambda_t(\dif s)=\int_0^{\infty}P_sf(x) \,g_{_t}(\alpha,s)\dif s,\qquad f\in \mathbf{B}.
\end{equation}
Then, $\set{Q_t,\, t\ge 0}$ is also a strong continuous contraction semigroup on $\mathbf{B}$ and its infinitesimal generator satisfies that
\begin{equation*}
  -(-H)^\alpha f= \mathcal{M}f=\int_0^{\infty}\,(P_sf-f)\rho(\dif s)=\frac{\alpha}{\Gamma(1-\alpha)}\int_0^{\infty}\,\frac{ P_sf-f}{s^{1+\alpha}}\dif s,\qquad f\in \mathcal{D}(L).
\end{equation*}
Moreover, $\mathcal{D}(H)$ is a core of $\mathcal{M}$ which means that $\mathcal{D}(H)\subset \mathcal{D}(\mathcal{M})$ and the closure of $\mathcal{M}|_{\mathcal{D}(H)}$, the restriction of $\mathcal{M}$ to $\mathcal{D}(H)$, equals $\mathcal{M}$. In fact,
$H$ can be replaced by a more general operator.

   The transformation of $\set{P_t}$ to $\set{Q_t}$ is called subordination and
we  refer to \cite{sato 13} and the references therein.  

The main results of the present paper are 
gradient   and stability estimates of the heat kernels
associated with  the fractional power for uniformly elliptic operators.

\begin{thm}\label{uniform ell op bound 1}
Suppose that $\mathbf{D}$ is a bounded $C^2$ domain in $\Rnum^{d}$ and $H=\nabla(a(x)\nabla)$ on $\mathbf{D}$ where the matrix $a(x)$ has $C^{\beta},\,\beta\in (0,1)$ entries, and there exists a constant $\lambda\geqslant 1$ such that $\lambda^{-1}\mathrm{Id}_{d}\leqslant a(x) \leqslant \lambda \mathrm{Id}_{d}$. Then the heat kernel $q(t,x,y)$ of the fractional power of $H$, i.e.,$\mathcal{M}=-(-H)^{\alpha}$, exists and  has the following gradient estimates: \begin{equation}\label{gradient gaussian estimate}
   \abs{\nabla_x q(t,x,y)}\le {c_1}\Big(t^{-\frac{d+1}{2\alpha}}\wedge {\frac{ t }{\abs{x-y}^{d+1+2\alpha}}}\Big),\qquad\forall  (t,x,y)\in (0,\infty)\times \mathbf{D}\times \mathbf{D},
   \end{equation}where $c_1=c_1(\mathbf{D},d,\alpha,\beta)$ is a strictly positive constant.
\end{thm}
\begin{thm}\label{stability thm}
Suppose that $\mathbf{D}=\Rnum^d$ and $H=\nabla(a(x)\nabla),\,\tilde{H}=\nabla(\tilde{a}(x)\nabla)$ on $\mathbf{D}$ with measurable coefficients. If there exists a constant $\lambda\geqslant 1$ such that $\lambda^{-1}\mathrm{Id}_{d}\leqslant a(x)  \leqslant \lambda \mathrm{Id}_{d}$ and $\lambda^{-1}\mathrm{Id}_{d}\leqslant \tilde{a}(x) \leqslant \lambda \mathrm{Id}_{d}$, then their subordinated semigroups $\set{Q_t},\,\set{\tilde{Q}_t}$ and the corresponding heat kernels $q(t,x,y),\,\tilde{q}(t,x,y)$ satisfy the following stability estimate: there exist bounded,   continuous functions $F_1(t,z),\,F_2(t,z)$ on $(0,\infty)\times (0,\infty)$ with $\lim_{z\to 0}F_i(t,z)=0,\,i=1,2$ for each $t>0$, 
\begin{align}
  \norm{Q_t-\tilde{Q}_t}_{p}&\leqslant F_1(t,\,\norm{a-\tilde{a}}_{L^2_{loc}}),\qquad \forall p\in [0,\infty]\label{stability qt}\\
  \abs{q(t,x,y)-\tilde{q}(t,x,y)}&\leqslant F_2(t,\,\norm{a-\tilde{a}}_{L^2_{loc}}).\label{stability kernel}
\end{align} where
\begin{align*}
   \norm{a-\tilde{a}}_{L^2_{loc}}=\sup_{\mathbf{k}\in \Znum^d}\sum_{i,j=1}^d \norm{a_{ij}-\tilde{a}_{ij}}_{L^2(D_{\mathbf{k}})},\\
   D_{\mathbf{k}}=\set{\mathbf{x}\in \Rnum^d:\,\abs{\mathbf{x}-\mathbf{k}}<2\sqrt{d}}.
\end{align*}
See below (\ref{stability qt norm}) and (\ref{stability kernel}) for the explicit expression of the functions $F_1(t,z),\,F_2(t,z)$.
\end{thm}
\begin{remark}
 Similarly, we can show that the inequality (\ref{stability qt}) is still valid for $\mathbf{D}$ is a bounded $C^1$-smooth domain in $\Rnum^{d}$.
\end{remark}
\section{Proof of the main theorems}
\subsection{Preliminaries:  Heat kernel of the subordination semigroup}
The asymptotic behaviors of $g(\alpha, s)$ when $s\to 0$ and when $s\to\infty$ have been known. See for example \cite{sato 13} Equality (14.35) for $s\to 0$ and Equality (14.37) for $s\to \infty$. For the convenience of readers we
  recall  these asymptotic formulae in following proposition.
\begin{proposition}\label{linnik storhod min}
The function $g(\alpha,\, s)$ has the following asymptotic formulae:
\begin{eqnarray}
     && g(\alpha,\, s) \sim Ks^{-\frac{2-\alpha}{2-2\alpha}}\exp(-As^{-\frac{\alpha}{1-\alpha}}),\quad \text{for} \quad s\to 0+,\label{asy formulae 1}\\
 & &    g(\alpha,\, s) \sim Bs^{-1-\alpha},\quad \text{for} \quad s\to \infty.\label{asy formulae 2}
\end{eqnarray}
   where $A,\,K$ and $B$ are constants only depending on $\alpha$.
\end{proposition}

It is known that if $\set{P_t}$ has a positive kernel then so does $\set{Q_t}$, see for example Lemma 3.4.1 of
\cite{davies 1990} and Lemma 5.4 of \cite{grigor 03}. We restate it as the following proposition.
\begin{prop}\label{heat estamate 0}
Suppose that $\mathbf{D}=\Rnum^d$ and $H=\nabla(a(x)\nabla),\,\tilde{H}=\nabla(\tilde{a}(x)\nabla)$ on $\mathbf{D}$ with measurable coefficients and suppose  there exists a constant $\lambda\geqslant 1$ such that $\lambda^{-1}\mathrm{Id}_{d}\leqslant a(x)  \leqslant \lambda \mathrm{Id}_{d}$ and $\lambda^{-1}\mathrm{Id}_{d}\leqslant \tilde{a}(x) \leqslant \lambda \mathrm{Id}_{d}$.
Then $\set{Q_t}$ also has a positive kernel $q(t,x,y)$ on $(0,\infty)\times \mathbf{D} \times \mathbf{D}$ such that
    \begin{equation}\label{qtxy}
      q(t,x,y)= \int_ 0^\infty  p(s,x,y) \,\lambda_t(\dif s)=\int_ 0^\infty  p(t^{\frac{1}{\alpha}}s,x,y)g(\alpha,\, s) \dif s.
    \end{equation}
\end{prop}
\begin{proof}
Since $p(t,x,y)$ is the positive kernel of $\set{P_t}$, it follows from Eq.(\ref{qtof f})
\begin{align*}
    q(t,x,y)&=\int_ 0^\infty p(s,x,y) \,\lambda_t(\dif s)\nonumber \\
 &=t^{-\frac{1}{\alpha}}\int_ 0^\infty  p(s,x,y)g(\alpha,\, t^{-\frac{1}{\alpha}} s) \dif s\qquad \text{(by Eq.(\ref{g t x}))}\nonumber \\
   &=\int_ 0^\infty  p(t^{\frac{1}{\alpha}}s',x,y)g(\alpha,\, s') \dif s'.
\end{align*}
This completes the proof.
\end{proof}
As a direct corollary of the above two propositions, we have the following results.
\begin{thm}\label{t.heat}
 Suppose that $\mathbf{D}$ is a domain in $\Rnum^{d}$ and $H=\nabla(a(x)\nabla)$ on $\mathbf{D}$ with measurable coefficients. If there exists a constant $\lambda\geqslant 1$ such that $\lambda^{-1}\mathrm{Id}_{d}\leqslant a(x)  \leqslant \lambda \mathrm{Id}_{d}$, then the heat kernel $q(t,x,y)$ of the fractional power of $H$, i.e.,$\mathcal{M}=-(-H)^{\alpha}$, has the following Nash's H\"{o}lder estimates: there are constants $c=c(d,\lambda,\alpha)>1$ and $\gamma=\gamma(d,\lambda)\in (0,1)$ such that
\begin{empheq}[left=\empheqlbrace]{align}
 &     \frac{1}{c}\Big(t^{-\frac{d}{2\alpha}}\wedge {\frac{ t}{\abs{y-x}^{d+2\alpha}}}\Big)\le q(t,x,y)\le {c}\Big(t^{-\frac{d}{2\alpha}}\wedge {\frac{ t}{\abs{y-x}^{d+2\alpha}}}\Big)\\
&  \abs{q(t,x,y)-q(t,x_1,y_1)}\leqslant c t^{-\frac{2d-\gamma}{2\alpha}}\big(\abs{x-x_1}\vee \abs{y-y_1}\big)^{\gamma},
 \end{empheq}  for all $t>0$ and $ (x,y),\,(x_1,y_1)\in \mathbf{D}\times \mathbf{D}$.
\end{thm}
\begin{proof} The first result is already known, see \cite[Lemma 5.4]{grigor 03}.
It is a consequence of \eqref{qtxy} and the following estimates
 \begin{equation}\label{global aronson}
      \frac{1}{M}t^{-\frac{d}{2}}\exp\set{-\frac{ M \abs{y-x}^2}{t}}\le p(t,x,y)\le {M}t^{-\frac{d}{2}}\exp\set{-\frac{\abs{y-x}^2}{M t}}\,.
   \end{equation}
The second inequality  follows from Eq.(1.3) of \cite{chz 98}.
   We shall not provide details since it will be similar to the proof that we
   present below.
\end{proof}
\subsection{Proof of the main theorems}
Theorem~\ref{uniform ell op bound 1} is a corollary of the following proposition.
\begin{prop}\label{heat estamate}
 Suppose that $\mathbf{D}$ is a domain in $\Rnum^{d}$. If there are two constants $M>0$ and $\ell\ge 0 $ such that $\abs{\nabla_x p(t,x,y)}$ has an upper bound
   \begin{equation}\label{gradient global aronson 1}
      \abs{\nabla_x p(t,x,y)}\le {M}t^{-\frac{\ell }{2}}\exp\set{-\frac{\abs{x-y}^2}{M t}},\qquad \forall (t,x,y)\in (0,\infty)\times \mathbf{D}\times \mathbf{D},
   \end{equation}then there is a strictly positive constant $c_1=c_1(M,\ell,\alpha)$ such that
     \begin{equation}\label{gradient gaussian bogdan-jakubowski 1}
    \abs{\nabla_x q(t,x,y)}\le {c_1}\Big(t^{-\frac{\ell }{2\alpha}}\wedge {\frac{ t }{\abs{x-y}^{\ell +2\alpha}}}\Big),\qquad\forall  (t,x,y)\in (0,\infty)\times \mathbf{D}\times \mathbf{D}.
   \end{equation}
\end{prop}
\begin{proof}
  We shall divide the proof into several steps. The idea is similar to the proof of Lemma 5.4 of \cite{grigor 03}.

Step 1. It follows from Lebesgue's dominated theorem that condition (\ref{gradient global aronson 1}) implies that one can take the derivative under the integral sign in Eq.(\ref{qtxy}), i.e.,
\begin{align*}
   \nabla_x q(t,x,y)=\int_ 0^\infty  \nabla_x p(t^{\frac{1}{\alpha}}s,x,y)g(\alpha,\, s) \dif s.
\end{align*}
Hence   inequality (\ref{gradient global aronson 1}) imply that for all $(t,x,y)\in (0,\infty)\times \mathbf{D}\times \mathbf{D}$,
  \begin{align}
     \abs{\nabla_x q(t,x,y)}&\leqslant \int_ 0^\infty \abs{\nabla_x  p(t^{\frac{1}{\alpha}}s,x,y)}g(\alpha,\, s) \dif s  \nonumber\\ 
      & \leqslant Mt^{-\frac{\ell}{2\alpha}}  \int_ 0^\infty s^{-\frac{\ell} 2}\exp\set{-\frac{  \abs{x-y}^2}{ M t^{\frac{1}{\alpha}}s}} g(\alpha, s)\dif s . \label{derivative under int}
  \end{align}
  Since the exponential function in the above integrand is less than one, we have that
  \begin{align*}
   & \int_ 0^\infty s^{- \frac{\ell}2  }\exp\set{-\frac{2 \abs{x-y}^2}{ M t^{\frac{1}{\alpha}}s}} g(\alpha, s)\dif s \le \int_ 0^\infty s^{-
    \frac{\ell} 2 } g(\alpha, s)\dif s\\
    &\qquad\qquad \le \int_ 0^1 s^{- \frac{\ell} 2  } g(\alpha, s)\dif s+1
  \end{align*}
since $g(\alpha, s)$ is  positive   and $\int_0^\infty g(\alpha, s) ds=1$.
Using  the asymptotic formula Eq.(\ref{asy formulae 1}) we also see that
$\int_ 0^1 s^{- \frac{\ell} 2  } g(\alpha, s)\dif s$ is finite. Thus we have
  \begin{equation}\label{neq new}
    \abs{\nabla_x  q(t,x,y)}\le  M t^{-\frac{\ell}{2\alpha}}
  \end{equation}

Step 2.  It is easy to see  from Eqs
(\ref{asy formulae 1})-(\ref{asy formulae 2}) that there exists a constant $\hat{c}:=\hat{c}(\alpha)>0$ such that
  \begin{equation}\label{ineq ga s}
     g(\alpha, s)\le \hat{c} s^{-1-\alpha},\qquad \forall s\in[0,\infty).
  \end{equation}

Substituting this  inequality 
into Eq.  (\ref{derivative under int}), we obtain that
  \begin{align}
     \abs{\nabla_x q(t,x,y)}
      &\leqslant \check{c}Mt^{-\frac{\ell}{2\alpha}}  \int_ 0^\infty s^{-\frac\ell 2-1-\alpha}\exp\set{-\frac{   \abs{x-y}^2}{ M t^{\frac{1}{\alpha}}s}}\dif s \nonumber\\
     &= \check{c}Mt^{-\frac{\ell}{2\alpha}}  \int_ 0^\infty r^{ \frac\ell 2 +\alpha-1}\exp\set{-\frac{  \abs{x-y}^2}{ M t^{\frac{1}{\alpha}}}r} \dif r  \qquad \text{(let $r=\frac{1}{s})$} \nonumber\\
     &=\check{c}Mt^{-\frac{\ell}{2\alpha}} \frac{\Gamma(\alpha + \frac\ell 2 )}{\big(  \abs{x-y}^2/Mt^{\frac{1}{\alpha}} \big)^{\alpha+ \frac\ell 2 }}
     \nonumber\\
     &= \check{c}\tilde M   {\frac{ t }{\abs{x-y}^{\ell+2\alpha}}},
     \label{qtxy le 2}
  \end{align} where $\Gamma(\cdot)$ is the Gamma function.
By putting the inequalities (\ref{neq new}) and (\ref{qtxy le 2}) together, we obtain the desired gradient estimate (\ref{gradient gaussian bogdan-jakubowski 1}).
\end{proof}

\noindent{\it Proof of Theorem~\ref{uniform ell op bound 1}.\,}
 { For the uniformly elliptic operator $H$, it is known that its kernel has the following gradient estimate, see for example \cite{zhangqi 96} Inequality (1.6) or \cite{zhangqi 97} Inequality (1.4),
 \begin{align*}
   \abs{\nabla_x p(t,x,y)}\le {M}t^{-\frac{d+1 }{2}}\exp\set{-\frac{\abs{x-y}^2}{M t}},\qquad \forall (t,x,y)\in (0,\infty)\times \mathbf{D}\times \mathbf{D}.
 \end{align*}
 }
 Hence it follows from Proposition~\ref{heat estamate} that Theorem~\ref{uniform ell op bound 1} holds.
{\hfill\large{$\Box$}}\\

\noindent{\it Proof of Theorem~\ref{stability thm}.\,}  We shall divide the proof into several steps.

 Step 1. Proposition \ref{heat estamate}
  can be rewritten as the following:  for a positive function $f(t,x,y)$, if there are two constants $M>0$ and $\ell\ge 0 $ such that $f(t,x,y)$ has an upper bound
   \begin{equation*}
      f(t,x,y)\le {M}t^{-\frac{\ell}{2}}\exp\set{-\frac{\abs{x-y}^2}{M t}},\qquad \forall (t,x,y)\in (0,\infty)\times \mathbf{D}\times \mathbf{D},
   \end{equation*}then there is a strictly positive constant $c_1=c_1(M,\ell,\alpha)$ such that
    \begin{equation}\label{old gaussian bogdan-jakubowski 1}
   \int_{0}^{\infty}f(t^{\frac{1}{\alpha}}s,x,y)g(\alpha, s)\dif s\le {c_1}\Big(t^{-\frac{\ell}{2\alpha}}\wedge {\frac{ t}{\abs{x-y}^{\ell+2\alpha}}}\Big),\qquad\forall  (t,x,y)\in (0,\infty)\times \mathbf{D}\times \mathbf{D}.
   \end{equation}

It follows form Theorem~1.2 of \cite{chz 98} that
\begin{align*}
   \abs{p(t,x,y)-\tilde{p}(t,x,y)}\leqslant c' t^{-\frac{d}{2}}\exp\set{-\frac{\abs{x-y}^2}{c t}} (t\wedge 1)^{-\gamma} \norm{a-\tilde{a}}_{L^2_{\loc}}^{\delta}\\
   \leqslant c' t^{-\frac{d}{2}}\exp\set{-\frac{\abs{x-y}^2}{c t}} (t^{-\gamma}+ 1) \norm{a-\tilde{a}}_{L^2_{\loc}}^{\delta} ,
\end{align*}where $c>0,c'>1, \gamma\in (0,1),\delta\in (0,1)$ depend only on $d$ and $\lambda$.
Hence it follows from (\ref{qtxy}) and the inequality (\ref{old gaussian bogdan-jakubowski 1}) that
\begin{align*}
   \abs{q(t,x,y)-\tilde{q}(t,x,y)}&\leqslant \int_0^{\infty}\abs{p(t^{\frac{1}{\alpha}}s,x,y)-\tilde{p}(t^{\frac{1}{\alpha}}s,x,y)}g(\alpha, s)\dif s\nonumber\\
   &\leqslant \tilde{c}\norm{a-\tilde{a}}_{L^2_{\loc}}^{\delta} \Big[t^{-\frac{d}{2\alpha}}\wedge {\frac{ t}{\abs{x-y}^{d+2\alpha}}}+t^{-\frac{d+2\gamma}{2\alpha}}\wedge {\frac{ t}{\abs{x-y}^{d+2\gamma+2\alpha}}}  \Big].
\end{align*}
On the other hand,  the first inequality in Theorem \ref{t.heat}
 gives the following bound:
\begin{equation*}
   \abs{q(t,x,y)-\tilde{q}(t,x,y)}\leqslant {c_1} \Big(t^{-\frac{d}{2\alpha}}\wedge {\frac{ t}{\abs{x-y}^{d+2\alpha}}}\Big)\,,  \quad  \forall  (t,x,y)\in (0,\infty)\times \mathbf{D}\times \mathbf{D},
\end{equation*}where $c_1=c_1(d,\lambda,\alpha)$ is a constant.

Combining the above two bounds together, we prove \eqref{stability kernel}
with the choice
\begin{equation}\label{stability kernel}
  F_2(t,z)=c \Big(t^{-\frac{d}{2\alpha}}\wedge {\frac{ t}{\abs{x-y}^{d+2\alpha}}}\Big) \min\set{1,\, t^{-\frac{\gamma}{\alpha}} \wedge \abs{x-y}^{-2\gamma} z^{\delta} }.
\end{equation}

Step 2. It follows from (\ref{qtof f}) that

\begin{align*}
   Q_tf(x)-\tilde{Q}_tf(x)=\int_0^{\infty}\big(P_sf(x)-\tilde{P}_sf(x)\big) \,g_{_t}(\alpha,s)\dif s\,.
\end{align*}
Thus it follows from Minkowski's integral inequality that for any $p\in [0,\infty]$,
\begin{align*}
   \norm{ Q_t -\tilde{Q}_t}_{p} \leqslant \int_0^{\infty}\norm{P_s -\tilde{P}_s}_{p} \,g_{_t}(\alpha,s)\dif s.
\end{align*}
It follows from Theorem 1.1 of \cite{chz 98} that
 \begin{align*}
  \norm{P_t -\tilde{P}_t}_{p}\leqslant  c'\cdot (t\wedge 1)^{-\gamma} \norm{a-\tilde{a}}_{L^2_{\loc}}^{\delta}\leqslant c'\cdot (t^{-\gamma}+ 1) \norm{a-\tilde{a}}_{L^2_{\loc}}^{\delta}.
 \end{align*}
 Hence we have that
 \begin{align*}
  \norm{ Q_t -\tilde{Q}_t}_{p}& \leqslant  c'\cdot \norm{a-\tilde{a}}_{L^2_{\loc}}^{\delta} \int_0^{\infty}\Big((t^{\frac{1}{\alpha}}s)^{-\gamma}+ 1\Big)g(\alpha,s)\dif s\nonumber\\
  &\leqslant  \hat{c} \cdot \norm{a-\tilde{a}}_{L^2_{\loc}}^{\delta}(1+ t^{-\frac{\gamma}{\alpha}}).
 \end{align*}
 Since $Q_t$ is also a contraction semigroup, we prove \eqref{stability qt}
 y the choice
 \begin{equation}\label{stability qt norm}
    F_1(t,z)=\min\set{2,\,c (1+ t^{-\frac{\gamma}{\alpha}})z^{\delta}}.
 \end{equation}
 This completes the proof of Theorem 2.
{\hfill\large{$\Box$}}\\

\vskip 0.2cm {\small {\bf  Acknowledgements}:
Y. Chen is supported by China Scholarship Council (201608430079) and Hubei Provincial NSFC (2016CFB526);
Y. Hu is partially supported by Simons Foundation (209206);
Z. Wang is supported by Mathematical Tianyuan Foundation of China (11526117) and Zhejiang Provincial NSFC (LQ16A010006).}


\end{document}